\documentclass[12pt]{amsart}

\usepackage{graphicx}					
\usepackage{amssymb}
\usepackage{mathtools}
\usepackage{xcolor}
\usepackage{mathabx}
\usepackage{amsmath}
\usepackage{amssymb}
\usepackage{amsthm}
\usepackage{url}
\usepackage{tkz-euclide}

\newtheorem{theorem}{Theorem}[section]

\newtheorem{lemma}[theorem]{Lemma}
\newtheorem{corollary}[theorem]{Corollary}

\newtheorem{claim}{Claim}[theorem]

\newtheorem*{sg}{Sylvester-Gallai Theorem}
\newtheorem*{kelly}{Kelly's Theorem}
\newtheorem*{grt}{Geometric Ramsey Theorem}
\newtheorem*{ghjt}{Geometric Hales-Jewett Theorem}

\newcommand{\del}{\setminus}
\newcommand{\con}{/}
\DeclareMathOperator{\si}{si}
\DeclareMathOperator{\cl}{cl}
\DeclareMathOperator{\HJ}{HJ}
\DeclareMathOperator{\R}{R}
\DeclareMathOperator{\AG}{AG}
\DeclareMathOperator{\PG}{PG}
\DeclareMathOperator{\GF}{GF}
\newcommand{\bF}{\mathbb{F}}

\begin{document}

\title{Unavoidable Flats in Matroids Representable over Prime Fields}

\author{Jim Geelen}
\address{Department of Combinatorics and Optimization, University of Waterloo, Waterloo, Canada} \author{Matthew E. Kroeker}
\address{Department of Combinatorics and Optimization, University of Waterloo, Waterloo, Canada} 

\thanks{This research was partially supported by grants from the Office of Naval Research [N00014-10-1-0851] and NSERC [203110-2011] and by an NSERC Postgraduate Scholarship [Application No. PGSD3 - 547508 - 2020]}

\begin{abstract}
We show that, for any prime $p$ and integer $k \geq 2$, a simple $\GF(p)$-representable matroid with sufficiently high rank has a rank-$k$ flat which is either independent in $M$, or is a projective or affine geometry.  As a corollary we obtain a Ramsey-type theorem for $\GF(p)$-representable matroids. For any prime $p$ and integer $k\ge 2$, if we $2$-colour the elements in any 
simple $\GF(p)$-representable matroid with sufficiently high rank, then there is a monochromatic flat with rank $k$.
\end{abstract}

\maketitle

\sloppy

\section{introduction}

In 1944, Gallai proved the following classical result~\cite{Gallai}, originally conjectured by Sylvester fifty years prior~\cite{Sylvester}.

\begin{sg}
Every rank-$3$ real-representable matroid has a two-point line.
\end{sg}

Much attention has been given to the question of how to meaningfully generalize the Sylvester-Gallai Theorem, as well as whether similar phenomena occur in more abstract geometric settings. The following result of Kelly~\cite{Kelly} is a particularly famous example of the latter.

\begin{kelly}
Every rank-$4$ complex-representable matroid has a two-point line.
\end{kelly}

Note that Kelly's rank bound is best-possible: for example, the ternary affine plane is well known to be complex-representable. For more information on the Sylvester-Gallai Theorem and its various extensions and generalizations, see Borwein and Moser's survey~\cite{BoMo}. In this paper, we are  motivated by Sylvester-Gallai-type problems which assert the existence independent flats of rank greater than two. For instance, under what conditions is a matroid guaranteed to contain a three-point plane? For the case of real-representability, Bonnice and Edelstein~\cite{BoEd} deduced the following result as a corollary of Hansen's Theorem~\cite{Hansen}.

\begin{theorem}[\cite{BoEd}, \cite{Hansen}]\label{hansen}
For an integer $k \geq 2$, if $M$ is a simple rank-$(2k-1)$ real-representable matroid, then $M$ contains a rank-$k$ independent flat.
\end{theorem}

The bound on the rank in Theorem~\ref{hansen} is tight, as can be seen by considering the direct sum of $k-1$ lines. This result was extended to complex-representable matroids by Barak, Dvir, Wigderson and Yehudayoff~\cite{Barak} with weaker bounds. Those bounds have since been tightened (see~\cite{Dvir,GeKr}), but there is still room for further improvement.

Projective geometries show that the Sylvester-Gallai theorem does not hold for matroids in general nor even for matroids representable over finite fields. However,  Murty~\cite{Murty} showed that any simple rank-$r$ matroid with fewer than $2^{r}-1$ points has a two-point line, and additionally that the binary projective geometry is the smallest such matroid with no two-point line. A stronger result for matroids representable over prime fields was proved by Bhattacharyya, Dvir, Shpilka, and Saraf~\cite{BDSS} who showed that, for a prime $p$ and sufficiently large integer $r$, any rank-$r$ $\GF(p)$-representable matroid with no two-point line has at least $p^{\Omega(r)}$ points.

In this paper we prove an analogue to Theorem~\ref{hansen} for the class of $\GF(p)$-representable matroids, where $p$ is prime. Our approach is to determine the unavoidable rank-$k$ flats in $\GF(p)$-representable matroids with sufficiently high rank. The following is our main result. 

\begin{theorem}\label{unavoidable}
For every prime $p$ and integer $k \geq 1$ there is an integer $N_{p}(k)$ such that, if $M$ is a simple $\GF(p)$-representable matroid with $r(M) \geq N_{p}(k)$, then $M$ contains a rank-$k$ flat $F$ such that either $F$ is independent in $M$, $M|F \cong \AG(k-1,p)$, or $M|F \cong \PG(k-1,p)$. 
\end{theorem}

Since affine and projective geometries both contain $p$-point lines, if we preclude $p$-point lines, then independent rank-$k$ flats become unavoidable. We can also avoid  affine and projective geometries by considering matroids with girth at least five; note that $\AG(n-1,2)$ has girth four.

\begin{corollary}
For every prime $p$ and integer $k \geq 1$, if $M$ is a simple $\GF(p)$-representable matroid with sufficiently high rank and girth at least five, then $M$ contains a rank-$k$ independent flat.
\end{corollary}

Another consequence of Theorem~\ref{unavoidable} is a Ramsey theorem for $\GF(p)$-representable matroids. 
\begin{corollary}\label{Ramsey}
For a prime $p$ and integer $k\ge 1$, if $M$ is a simple $\GF(p)$-representable matroid with sufficiently large rank then, then for any $2$-colouring of the points of $M$, there is a monochromatic rank-$k$ flat.
\end{corollary}

Note that, by Theorem~\ref{unavoidable}, to prove Corollary~\ref{Ramsey} it suffices to consider the cases where
$M$ is a large independent set or $M$ is a high rank affine or projective geometry.
For independents sets we get a monochromatic flat by majority, and 
for projective and affine geometries we use  
known Ramsey theorems; see Section~3.

\section{Building on Kelly's Proof}

We start by reviewing Kelly's proof~\cite{Kelly} that every rank-$4$ complex-representable matroid has a two-point line. His proof uses a result of Hirzebruch~\cite{Hirzebruch} that every rank-$3$ complex-representable matroid has a line with at most three points. Let $e$ be an element in a simple rank-$4$ complex-representable matroid $M$.  We assume by way of contradiction that $M$ has no $2$-point line. By Hirzebruch's result, 
$M\con e$ contains a line $L$ with exactly  three points, say $p_1$, $p_2$, and $p_3$.
Let $N$ be the restriction of $M$ to the set $L\cup\{e\}$. Then $N$ is a simple rank-$3$ matroid, with no $2$-point line, and $N$ 
comprises three copunctual lines $p_1\cup\{e\}$, $p_2\cup\{e\}$, and $p_3\cup\{e\}$. Kelly's proof is completed by the following result, which is implicit in~\cite{Kelly}. 
\begin{lemma}\label{kelly} For a field $\bF$,
let $N$ be an $\bF$-representable rank-$3$ matroid comprising three co-punctual lines $L_1$, $L_2$, and $L_3$.
If $N$ has no $2$-point line, then $|L_1|=|L_2|=|L_3|$ and $|L_1|-1$ is divisible by the characteristic of $\bF$, and, hence, $\bF$ has positive characteristic.
\end{lemma}

\begin{proof}[Proof sketch.]
Let $e$ be the common point of $L_1$, $L_2$, and $L_3$.
Any two points $e_1\in L_1\setminus\{e\}$ and $e_2\in L_2\setminus \{e\}$ are collinear with a third point $e_3\in L_3\setminus\{e\}$. By fixing a particular choice of element $e_2$ and varying $e_1$ we see that $|L_3|\ge |L_1|$. Then, by symmetry, we see that $L_1$, $L_2$, and $L_3$ all have the same size. 
Fix two distinct elements $a,b\in L_2\setminus\{e\}$. 
Consider the $2$-regular bipartite graph $H$ with bipartition $(L_1,L_3)$ such that $x\in L_1$ is adjacent to $y\in L_3$ if 
either $\{x,a,y\}$ or $\{x,b,y\}$ is a triangle. Thus $H$ is a disjoint union of even cycles. Consider one of these cycles, say $C$.
A straightforward argument considering the representation of $M|(\{a,b\}\cup V(C))$ reveals that the characteristic of $\bF$ is $\frac{1}{2} |V(C)|$.
Since this holds for each component of $H$, we see that the characteristic of $\bF$ divides $|L_1\setminus\{e\}|$, as required.
\end{proof} 

The matroid $N|(\{a,b\}\cup L_1\cup L_3)$, in the proof sketch, is an example of a ``Reid Geometry''. The details regarding the representation of Reid Geometries, which were omitted in our proof sketch, are given explicitly by Kung~\cite[Proposition 2.2]{Kung}. 

In the case that $\bF$ is a finite field of size $p$, where $p$ is prime, the lines $L_1$, $L_2$, and $L_3$, in Lemma~\ref{kelly}, must have length exactly $p+1$, since longer lines are not
GF$(p)$-representable.
\begin{lemma}\label{reid1} For a prime $p$,
let $N$ be a simple $\GF(p)$-representable rank-$3$ matroid comprising three copunctual lines $L_1$, $L_2$, and $L_3$.
If $N$ has no $2$-point line, then the lines $L_1$, $L_2$, and $L_3$ each have exactly $p+1$ points.
\end{lemma}

While Kelly used Reid geometries to find two-point lines, we are interested in independent flats of arbitrary rank. In the next two lemmas, we derive a high-dimensional generalization of Lemma~\ref{reid1}  suitable for our purposes.  Note that throughout this paper, when we refer to a {\it{point}} in a matroid, we mean a rank-$1$ flat (as opposed to an element in the ground set).

\begin{lemma}\label{reid2} For an integer $m\ge 3$ and a prime $p$, 
let $e$ be an element of a  simple $\GF(p)$-representable matroid $M$ such that $\si(M\con e)$ is a connected $m$-element matroid.  If $e$ is not a coloop in $M$ and $M$ has no hyperplane disjoint from $e$ with fewer than $m$ points, then each of the lines of $M$ containing $e$ has length $p+1$. 
\end{lemma}

\begin{proof} 
Note that each hyperplane disjoint from $e$ contains exactly one element from each line containing $e$. Since $e$ is not a coloop, one of those lines, say $L_1$, has length at least $3$. 
Let $a_1$ and $a_2$ be distinct elements in $L_1\setminus \{e\}$ and let $b\in E(M)\setminus L_1$ be chosen arbitrarily.
There is a hyperplane $H$ that contains $a_1$ and $b$ and is disjoint from $e$.
Note that $M|H$ is isomorphic to ${\mathrm{si}}(M/e)$. Since ${\mathrm{si}}(M/e)$ is connected, there is a circuit $C$ of $M|H$ that contains both $a_1$ and $b$.
Let $c\in C\setminus \{a_1,b\}$ and let $L_2$ and $L_3$ be the lines spanned by $\{e,b\}$ and $\{e,c\}$ respectively. 
Let $X=C\setminus \{a_1,b,c\}$ and let $N$ denote the restriction of $M\con X$ to $L_1\cup L_2\cup L_3$. Note that 
$N$ is a simple rank-$3$ matroid and $L_1$, $L_2$, and $L_3$ are lines in $N$ that intersect at $e$. 

Consider any line $L$ of $N$ that does not contain $e$ and let $a$ and $b$ be distinct points of $L$.
Note that $X \cup \{a,b,e\}$ is independent in $M$. Extend $X \cup \{a,b,e\}$ to a basis
$B$ of $M$ and consider the hyperplane $F$ spanned by $B\setminus\{e\}$. 
This hyperplane intersects each of the lines through $e$ in a point, so $F$ contains one element from each of $L_1$, $L_2$, and $L_3$. Therefore $L$ is a $3$-point line in $N$.
Thus every line in $N$ disjoint from $e$  has at least three points. 

The lines spanned by $\{a_1,b\}$ and $\{a_2,b\}$, in $N$, both
intersect $L_3$ in a point distinct from $e$. Therefore $|L_3|\ge 3$ and, by symmetry, we also have $|L_2|\ge 3$. Thus $N$ has no two-point line. Therefore, by Lemma~\ref{reid1}, the lines $L_1$, $L_2$, and $L_3$ each have exactly $p+1$ points. Then, by our choice of $b$, all lines through $e$ in $M$ have length $p+1$.
\end{proof}

We use Lemma \ref{reid2} iteratively to build affine geometries. 

\begin{lemma}\label{kelly2}
For integers $k\ge 2$  and $n \geq t\ge 2$ and a prime $p$, if $M$ is a matroid of rank at least $t+k-1$ such that every rank-$t$ flat in $M$ has at least $n$ points, then either
\begin{itemize}
\item[(i)] $M$ has an $\AG(k-1,p)$-restriction,
\item[(ii)] there is a rank-$(k-1)$ flat $F$ in $M$ such that every rank-$t$ flat in $M\con F$ has at least $n+1$ points, or
\item[(iii)] $M$ has a rank-$t$ flat $F$ such that $M|F$ is not connected.
\end{itemize}
\end{lemma}

\begin{proof}
It suffices to prove the result in the case that $M$ has rank equal to $t+k-1$. Let $C$ be a rank-$(k-1)$ flat in $M$.
Since $M\con C$ has rank $t$, we may assume that
$M\con C$ has exactly $n$ points since otherwise $(ii)$ holds. Moreover, we may assume that $M\con C$ is connected, since otherwise $(iii)$ holds.

\begin{claim} For each $e\in C$ and $f\in E(M)\setminus C$, the line spanned by $\{e,f\}$ has length $p+1$.
\end{claim}

\begin{proof}[Proof of claim.]
There is a rank-$t$ flat $F$ of $M$ such that $f\in F$ and $r(F\cup C)= r(F)+r(C)$. Note that $M|F\cong \si(M\con C)$, so $F$ has $n$ points and $M | F$ is connected.
Let $N$ denote the restriction of $M$ to the flat spanned by $F\cup\{e\}$. Note that $e$ is not a coloop of $N$ since otherwise by considering a hyperplane of $N$ containing $e$ we get a rank-$t$ flat with fewer than $n$ points. Then, by Lemma~\ref{reid2}, every line of $N$ through $e$ has length $p+1$. In particular, the line spanned by $\{e,f\}$ has length $p+1$.
\end{proof}

Let $\{e_1,\ldots,e_{k-1}\} $ be a basis of $C$, let $f\in E(M)\setminus C$, and for each $i\in\{1,\ldots,k-1\}$ let $S_i:=\cl(\{f,e_1,e_2,\ldots,e_i\})\setminus C$.
By the claim, $|S_1|=p$ and, for each $i\in\{1,\ldots,k-2\}$, each point in $S_i$ lifts to $p$ points in $S_{i+1}$. Therefore $|S_{i+1}|=p|S_i|$ and hence $|S_{k-1}|=p^{k-1}$.
It follows that $M|S_{k-1}\cong \AG(k-1,p)$.
\end{proof}

\section{Proof of the Main Theorem}

Lemma~\ref{kelly2} will do most of the work in finding the unavoidable flats in Theorem~\ref{unavoidable}, but two problems still remain to be solved: how do we find a flat with the desired size in the minor; and what do we do with an affine geometry restriction? We resolve both of these issues using Ramsey theory. The following is a consequence of Graham, Leeb and Rothschild's Ramsey theorem for vector spaces over a finite field~\cite{Graham}.

\begin{grt}
For each prime-power $q$ and positive integer $t$ there is an integer $\R_q(t)$ such that , if we two-colour the points in a rank-$\R_q(t)$ projective geometry over
$\GF(q)$, then there will be a monochromatic rank-$t$ flat.
\end{grt}

Our other Ramsey-theoretic tool is a corollary of the Hales-Jewett Theorem~\cite{Hales}.

\begin{ghjt}
For each prime-power $q$ and positive integers $t$ and $k$ there is an integer $\HJ_q(t,k)$ such that , if we $k$-colour the points in a rank-$\HJ_q(t,k)$ affine geometry over
$\GF(q)$, then there will be a monochromatic rank-$t$ flat.
\end{ghjt}

A consequence of this is that coextensions of huge affine geometries contain large affine geometries.

\begin{lemma}\label{affine}
For a prime-power $q$ and integers $k,m\ge 2$ and $n=\HJ_q(k,q^m)$, 
if $J$ is a $m$-element independent set in  $\GF(q)$-representable matroid $M$ and $M/J$ has an $\AG(n-1,q)$-restriction, then $M$ has an $\AG(k-1,q)$-restriction.
\end{lemma}

\begin{proof} We may assume that $M/J$ is isomorphic to $\AG(n-1,q)$.
The matroid $M$ has a $\GF(p)$-representation of the form
$$
A=\bordermatrix{ & J & \cr
 & I & B \cr
& 0& D},
$$ 
where $I$ denotes the $J\times J$ identity matrix and $D$ is a representation of $\AG(n-1,q)$.
We may further assume that the first row of $D$ is the all-ones vector.
Note that $B$ has at most $q^m$ distinct columns, and we assign each element of $M\con J$ a colour according to the associated column of $B$.
By the Geometric Hales-Jewett Theorem, there is a rank-$k$ flat $F$ of $(M/J)$ which is monochromatic under this colouring. By adding linear combinations of the first row of $D$ to the rows of $B$, we may assume that the restriction of $B$ to $F$ is the all-zero matrix. Therefore 
$M|F=(M\con J)|F\cong \AG(k-1,q)$, as desired.
\end{proof}

Lemma~\ref{affine} allows us to refine Lemma~\ref{kelly2}.

\begin{lemma}\label{restriction}
For integers $k \geq 2$, $n \geq t \geq 2$ and a prime $p$, there is an integer $N_{p}'(k,t,n)$ such that, if $M$ is a simple $\GF(p)$-representable matroid with $r(M) \geq N_{p}'(k,t,n)$, then $M$ has either an $\AG(k-1,p)$-restriction or a rank-$t$ flat which is either disconnected or has at most $n$ points.
\end{lemma}

\begin{proof}
Let $t \geq 2$, and observe that, for $n \geq \frac{p^{t}-1}{p-1}$, the integer $N_{p}'(k,t,n)$ exists for all $k \geq 2$. Now let $k \geq 2$, $t \leq n < \frac{p^{t}-1}{p-1}$, and assume that $N_{p}'(i,t,n+1)$ exists for all $i \geq 2$. Let $M$ be a simple $\GF(p)$-representable matroid with $$r(M) \geq t+k-1 + N_{p}'(\HJ_{p}(k,p^{k-1}),t,n+1).$$ We may assume that every rank-$t$ flat in $M$ is connected and has at least $n+1$ points. By Lemma~\ref{kelly2}, there is a rank-$(k-1)$ flat $F$ in $M$ such that every rank-$t$ flat in $M \con F$ has at least $n+2$ points. Moreover, every rank-$t$ flat in $M \con F$ is connected. By induction, $M \con F$ has an $\AG(\HJ_{p}(k,p^{k-1})-1,p)$-restriction. The result now follows by Lemma~\ref{affine}.
\end{proof}

We can now prove Theorem \ref{unavoidable}; to facilitate induction, we prove a slightly stronger result. 
\begin{theorem}
For every prime $p$ and integers $k,t \geq 1$ there is an integer $N_{p}(k,t)$ such that, if $M$ is a simple $\GF(p)$-representable matroid with $r(M) \geq N_{p}(k,t)$, then $M$ contains either a rank-$t$ independent flat or a flat $F$ such that $M|F \cong \AG(k-1,p)$ or $M|F \cong \PG(k-1,p)$.
\end{theorem}

\begin{proof}
For $p$ prime and $k \geq 1$, we prove the result for $N_{p}(k,1)=1$, and, for $t \geq 2$, we recursively define
$$N_{p}(k,t)=N_{p}'(R_{p}(k-1)+1, \, 2N_{p}(k,t-1),\, 2N_{p}(k,t-1)),$$ 
where, by convention, we take $R_{p}(0)=0$. We proceed by induction on $t$, where the base case holds because any point in a matroid forms a rank-$1$ independent flat. Now let $t \geq 2$, and assume that the assertion holds for $N_{p}(k,t-1)$ for every integer $k \geq 1$.

For $k \geq 1$, suppose $M$ is a simple $\GF(p)$-representable matroid with $r(M) \geq N_{p}(k,t)$.
Let $m=2N_{p}(k,t-1)$. Note that, if $F$ is a rank-$m$ flat with exactly $m$ points, then $M|F$ is not connected. Therefore,
applying Lemma~\ref{restriction} gives one of the following two cases.

{\flushleft \textbf{Case 1.} {\em There is a set $A \subseteq E(M)$ such that $M|A \cong \AG(R_{p}(k-1),p)$.}}

Let $F = \cl_{M}(A)$, and let $G \cong \PG(R_{p}(k-1),p)$ be a matroid with $F \subseteq E(G)$ and $G|F = M|F$. Then $E(G) \setminus A$ is a hyperplane in $G$, and so $G \del A \cong \PG(R_{p}(k-1)-1,p)$. By the Geometric Ramsey Theorem, there is a rank-$(k-1)$ flat $H$ in $G \del A$ such that either $H \subseteq E(M)$ or $H \subseteq E(G) \setminus E(M)$. Let $e \in A$ and let $F = E(M) \cap \cl_{G}(H \cup \{e\})$. Then $F$ is a rank-$k$ flat in $M$, and $M|F$ is either an affine or projective geometry.

{\flushleft \textbf{Case 2.} {\em There is a rank-$m$ flat $F$ of $M$ such that $M|F$ is not connected.} }

Let $N$ be a smallest component of $M|F$, and let $F'=F \setminus E(N)$. Then $r_{M}(F') \geq N_{p}(k,t-1)$, and so, by induction, we may assume that $M|F'$ has a rank-$(t-1)$ independent flat $I$. Then for any $e \in E(N)$, $I \cup \{e\}$ is a rank-$t$ independent flat in $M$.
\end{proof}

\end{document}